\documentclass[12pt,english]{article}
\usepackage{units}
\usepackage{amsthm}
\usepackage{amsmath}
\usepackage{amssymb}
\usepackage{graphicx}
\usepackage{esint}

\makeatletter
\numberwithin{equation}{section}
\numberwithin{figure}{section}
\theoremstyle{plain}
\newtheorem{thm}{\protect\theoremname}
  \theoremstyle{plain}
  \newtheorem{lem}[thm]{\protect\lemmaname}
  \theoremstyle{plain}
  \newtheorem{prop}[thm]{\protect\propositionname}
  \theoremstyle{remark}
  \newtheorem{rem}[thm]{\protect\remarkname}
  \theoremstyle{definition}
  
    \theoremstyle{plain}
   \newtheorem{corollary}[thm]{\protect\corollaryname}

\makeatother

\newcommand{\R}{{\mathbb{R}}}
\newcommand{\Z}{{\mathbb{Z}}}
\newcommand{\eps}{{\varepsilon}}
\usepackage{babel}
  \providecommand{\definitionname}{Definition}
  \providecommand{\lemmaname}{Lemma}
  \providecommand{\propositionname}{Proposition}
  \providecommand{\remarkname}{Remark}
\providecommand{\theoremname}{Theorem}
\providecommand{\corollaryname}{Corollary}
\begin{document}

\title{Convergence of numerical approximations to non-linear continuity equations with rough force fields}

\date{}

\author{F. Ben Belgacem \footnote{Laboratory of  partial differential equations (LR03ES04), Faculty of Sciences of Tunis, University of Tunis, El Manar, TUNISIA. 
Email:  belgacem.fethi@gmail.com, or fethi.benbelgacem@fst.rnu.tn },  
P-E Jabin \footnote{CSCAMM and Dept. of Mathematics, University of Maryland,
College Park, MD 20742, USA. Email: pjabin@umd.edu. P.--E.~{\sc Jabin} is partially supported by NSF Grant DMS 1312142 and 1614537 and by NSF Grant RNMS (Ki-Net) 1107444.}
 }
\maketitle
\begin{abstract}
We prove quantitative regularity estimates for the solutions to non-linear continuity equations and their discretized numerical approximations on Cartesian grids when advected by a rough force field. This allow us to recover the known optimal regularity for linear transport equations but also to obtain the convergence of a wide range of numerical schemes. Our proof is based on a novel commutator estimates, quantifying and extending to the non-linear case the classical commutator approach of the theory of renormalized solutions. 
\end{abstract}
\section{Introduction}
\subsection{The model}
One of the main goals of this article is to study the convergence of some simple numerical schemes for solution of the non-linear  equation
\begin{equation}
\partial_{t}u\left(t,x\right)+\textrm{div}\left(a\left(t,x\right)f\left(u\left(t,x\right)\right)\right)=0,\quad t\in\mathbb{R}_{+},\ x\in\mathbb{R}^{d},\label{eq:principale}
\end{equation}
in the case where the velocity field $a$ only belongs to $L^p_{loc} (\R_+,\ W^{1,p}(\R^d))$ and is hence not smooth. 

The density $u$ can model a large variety density of agents or objects from molecules to micro-organisms and individuals (in pedestrian models for instance). Eq. \eqref{eq:principale} combines a classical advection through the velocity $a$ with non-linear effects through the flux $f\in W^{1,\infty}\left(\mathbb{R},\mathbb{R}\right)$: It is hence a hybrid between a linear advection equation and a scalar conservation law.

A good example for $f$ is $f(u)=u\,(u_c-u)_+$ where $u_c$ is a critical density. Eq. \eqref{eq:principale} then ensures that $u\leq u_c$ at all times. This is an important feature when relatively large agents are considered in comparison to the length scale over which one calculates the density. In such a case, the maximal density of agents (where they all touch each other) may be of the same order of magnitude as the average density under consideration. This is usually the case for crowd motion models. We refer to \cite{Ma, MaRoSa} for examples of such congestion effects.

We will not consider any particular coupling between $a$ and $u$ in this article. Since we do not study uniqueness, we only need to assume that some Sobolev regularity is obtained on $a$. This makes our estimates compatible with a wide range of models. We give two examples; first coupling through the Poisson equation
\[
a=\nabla c,\quad -\Delta c=g(u).
\]
This is commonly used for so-called chemotaxis models, the dynamics of micro-organisms directed by chemical signals. In this case one considers only one chemical whose density is given by $c(t,x)$; the micro-organisms try to follow the gradient of the chemical. Such a model has been studied in \cite{DalPer}.

A variant of this is the Hamilton-Jacobi coupling
\[
a=\nabla c,\quad -\Delta c+|\nabla c|^2=g(u),
\]
which has been implemented for pedestrian models as in \cite{DiMaPi}. 

Eq. \eqref{eq:principale} includes, as a special case for $f=Id$, the classical continuity equation
\begin{equation}
\partial_{t}u\left(t,x\right)+\textrm{div}\left(a\left(t,x\right) \,u\left(t,x\right)\right)=0.\label{eq:continuity}
\end{equation}
The results presented here hence also apply to the case of \eqref{eq:continuity}. The non-linearity in \eqref{eq:principale} restricts many of the techniques that are available for \eqref{eq:continuity} which is one of the recurring difficulties in this article.

For simplicity, we call the general \eqref{eq:principale} the non-linear continuity equation and \eqref{eq:continuity} the linear continuity equation. This emphasizes their main structural difference but of course in most applications both  \eqref{eq:principale}  and  \eqref{eq:continuity} are part of a larger non-linear system which couples $a$ and $u$. 
\subsection{An example of application: Compactness and explicit regularity estimates for Eq. \eqref{eq:principale}}
The key difficulties in many of such complex, nonlinear models are possible instabilities in the density $\rho$: The main challenge is to {\em control how oscillations in $\rho$ can develop} in \eqref{eq:principale}, especially for rough velocity field  such as $u\in L^2_t H^2_x$ given by typical viscosity bounds. This makes the propagation of regularity on \eqref{eq:principale} at the center of our proposed work for convection models. 

Unfortunately, it is in general not possible to propagate any kind of Sobolev regularity on $\rho$, often leading to implicit or convoluted argument.

As a first illustration of the method introduced here, we present new explicit regularity estimates for solutions to \eqref{eq:principale}. Define the semi-norms for $0<\theta<1$
\[
\left\Vert u\right\Vert _{p,\theta}^{p}=\sup_{h\leq 1/2} |\log h|^{-\theta}\,\int_{\mathbb{R}^{2d}}{\displaystyle \frac{\mathbb{I}_{\left|x-y\right|\leq1}}{\left(\left|x-y\right|+h\right)^{d}}\left|u\left(t,x\right)-u\left(t,y\right)\right|^{p}dxdy.}
\]
Obviously the semi-norms are decreasing in $\theta$
\[
\|u\|_{p,\theta}\leq \|u\|_{p,\theta'}, \quad\mbox{if}\ \theta'\leq \theta,
\]
and they are only semi-norms as $\|u\|_{p,\theta}$ vanishes if $u$ is a constant. We may define the corresponding spaces
\[
W_{\log,\theta}^p=\{u\in L^p\;|\,\|u\|_{p,\theta}<\infty\}.
\]
Those semi-norms measure intermediary regularity ($log$ of a derivative) between $L^p$ spaces and Sobolev spaces $W^{s,p}$ as per the proposition
\begin{prop}
For any $s>0$, $0<\theta<1$, and any $1\leq p\leq \infty$, one has the embeddings $W^{s,p}\subset W_{\log,\theta}^p\subset L^p$ which are compact on any smooth bounded domain of $\R^d$. For $\theta=0$, $\|u\|_{p,0}\leq C\,\|u\|_{L^p}$. In addition for $p=2$,
\[\begin{split}
\|u\|_{2,\theta}^2+\|u\|_{L^2}^2&\sim \sup_h \int_{\R^d} \frac{\left|\log \left(\frac{1}{|\xi|}+h\right)\right|+1}{|\log h|^\theta}\,|{\cal F}\,u(\xi)|^2\,d\xi\\
&\leq \int_{\R^d} (\log (1+|\xi|))^{1-\theta}\,|{\cal F}\,u(\xi)|^2\,d\xi,
\end{split}\]
where ${\cal F}\,u$ denotes the Fourier transform of $u$. \label{propembedding}
\end{prop}
There has recently been an increase in the interest for such spaces which differ from classical Sobolev or $L^p$ spaces by a $\log$ scale; see for instance \cite{BoBrMi}.

The semi-norms are at the critical scale where regularity is propagated for Eq. \eqref{eq:principale} with 
\begin{thm}
Assume that $a$ belongs to the Besov space $L^1([0,\ T],\;B^1_{p,q}(\R^d))$ for some $p,\,q\geq 1$ with $\mbox{div}\,a\in L^\infty([0,\ T]\times\R^d)$. Any entropy solution $u\in L^\infty([0,\ T],\;L^{p^*}\cap L^{p^*,1})$ satisfies the regularity estimate for any $t\leq T$ and any
$\theta\geq \max(1/p^*,1-1/q)$
\[\begin{split}
\|u\|_{1,\theta}&\leq e^{C\, \|div\,a\|_{L^\infty}\,\|f'\|_{L^\infty} t}\,\Big(\|f'\|_{L^\infty}\, \|u\|_{L^\infty_t L^{p^*}_x}\,\int_0^t\|\mbox{div}\,a(s,.)\|_{p,p\,(\theta-1/p^*)}\,ds\\
&+C\,\|\nabla a\|_{L^1([0,\ T],\;B^0_{p,q}(\R^d))}\,\|f'\|_{L^\infty}\,\|u\|_{L^\infty_t L^{p^*,1}_x}+\|u^0\|_{1,\theta}\Big).
\end{split}
\]
This implies the simple estimate for $a\in L^1([0,\ T],\;W^{1,p}(\R^d))$ with $1<p\leq 2$ and $u\in L^\infty([0,\ T],\;L^1\cap L^r(\R^d))$ for $r>p^*$
\[\begin{split}
\|u\|_{1,1/2}&\leq e^{C\, \|div\,a\|_{L^\infty}\,\|f'\|_{L^\infty} t}\,\Big(\|f'\|_{L^\infty}\, \|u\|_{L^\infty_t L^{p^*}_x}\,\int_0^t\|\mbox{div}\,a(s,.)\|_{p,1-p/2}\,ds\\
&+C\,\|\nabla a\|_{L^1([0,\ T],\;L^p(\R^d))}\,\|f'\|_{L^\infty}\,\|u\|_{L^\infty_t L^1\cap L^{r}_x}+\|u^0\|_{1,1/2}\Big).
\end{split}\]
 \label{thregularity}
\end{thm}

For technical reason it is often more convenient to work with a smooth kernel in the definition of the semi-norms. Define
\begin{equation}
K_{h}\left(x-y\right)={\displaystyle \frac{\phi(x-y)}{\left(\left|x-y\right|+h\right)^{d}}},\label{defKh}
\end{equation}
for some smooth function $\phi$ with compact support in $B(0,2)$ and such that
$\phi=1$ inside $B(0,1)$. We can then take the variant definition
\[
\|u\|_{p,\theta}^p=\sup_{h\leq 1/2} |\log h|^{-\theta}\,\int_{\mathbb{R}^{2d}}K_{h}\left(x-y\right)\left|u\left(x\right) -u\left(y\right)\right|^{p}\textrm{d}x\,\textrm{d}y.
\]

A first rougher version of Th. \ref{thregularity} had been derived in \cite{BeJa}. The main estimate in the proof however was $L^2$ based, leading to non optimal estimates where $a\in W^{1,p}_x$ with $p\neq 2$. It was moreover essentially non compatible with a discretized setting such as the numerical schemes that we are mostly concerned with here. We have completely revisited the approach by identifying precisely the cancellations at the heart of Th. \ref{thregularity}. This lets us obtain the optimal regularity in a much more general setting and identify the critical Besov spaces for $a$.

Quantitative regularity estimates were first obtained for linear advection or continuity equations in \cite{CD}. The method there is based on bounds along the characteristics and very different from the one followed here. This characteristics method was later used in \cite{BBC, BoCr, CJ, CCS, JabMas} under various extended assumptions (singular integrals or force field with less than a derivative but with the right structure).

A more similar looking estimate has been obtained in \cite{BreJab} also at the PDE level. This last estimate relies on a duality method which is only compatible with linear continuity equations but can then be more carefully tailored to the problem. 

All those explicit estimates propagate some form of a $\log$ of a derivative, just like Th. \ref{thregularity}. In general this is the best that one can hope for in the presence of a Sobolev force field as was proved in \cite{ACM,Jab2}.

We further explain the connections between the present quantitative estimates and the classical theory of renormalized solutions for linear continuity equations when we state our commutator estimate in Section \ref{sec:commutator}. 
\section{The Numerical Scheme and main results}
We now turn to the main results of this article concerning the convergence of numerical schemes for Eq. \eqref{eq:principale}. Numerical schemes for advection equations with rough force fields have mostly been studied in the context of compressible fluid dynamics where the density satisfies the continuity equation with an only $H^1$ velocity field. We refer in particular to schemes for the compressible Stokes system with for instance \cite{EGHL,EGHL2, GHL}, or the Navier-Stokes system with \cite{GGHL, GHL2}.

Compressible Fluid dynamics models typically involve the linear continuity equation \eqref{eq:continuity} on the density. One of the major difficulties in proving the convergence of such numerical schemes is to obtain the compactness of the density. The convergence of schemes for the compressible Navier-Stokes system is in large part still an open question. We hope that the new quantitative estimates that we introduce can prove useful.
 
In addition to the linear continuity equation \eqref{eq:continuity}, Eq. \eqref{eq:principale} also contains the classical one-dimensional scalar conservation law
\begin{equation}
\partial_t u+\partial_x f(u)=0.\label{SCL}
\end{equation}
The well posedness theory for \eqref{SCL} is now well understood since the work of Kruzkov \cite{Kru}. The analysis of numerical schemes for conservation laws of which \eqref{SCL} is a very simple case is also classical and well-developed, we refer for example to \cite{Lev, Tad}. 

Eq. \eqref{SCL} exhibits shocks in finite time so that it only propagates up to $BV$ regularity or in general $W^{s,1}$ with $s<1$. One of the challenges of our study was to find regularity estimates which are compatible both with linear advection equations with rough force fields and with shocks from conservation laws.

In general our non-linear continuity equation could be seen as a conservation law with time and space dependent fluxes. Although there are some results for such systems with discontinuous fluxes, see e.g. \cite{AuPe, KaTo, Panov}, there do not seem to be applicable in a case such as ours where only Sobolev bounds are known in the absence of any other strong structure on the flux.

Before describing more in details the schemes that we consider, we want to emphasize here that we focus on schemes {\em on a Cartesian grid}. Non-Cartesian grids can be much more complicated from the point of view of the regularity as even for smooth velocity fields the discrete solutions may lose regularity (and not be in $BV$ for instance).
\subsection{Description of the schemes under consideration}
The discrete solution is given by a set of values $u_i^n$ representing an approximation of the continuous solution at time $t_n=n\,\delta t$ over the various points of the grid at $i=(i_1,\ldots,i_d)$ with $i_1\dots i_d=1\dots N$. We assume a grid length equal to $\delta x$ and denote by $x_i$ the points at the center of each mesh. 

We will make abundant use of the discrete $l^p$ norms which we normalize by the grid size
\[
\|u^n\|_{l^p}^p=\delta x^d\,\sum_{i\in \Z^d} |u^n_i|^p.
\]
For convenience we denote $i+[\tau]_k$, for $k=1\dots d$, the index where coordinate $i_k$ of $i$ is shifted by $\tau$. So for example $x_{i+[1]_k}$ is simply the center of the next mesh in direction $k$. 

This lets us easily define discrete Sobolev norms per
\[
\|a^n\|_{d,W^{1,p}}^p=\delta x^d\,\sum_{i\in \Z^d} \sum_{k=1}^d |a^n_i-a^n_{i+[1]_k}|^p.
\]
%
%
We consider fairly general explicit schemes of the form
\begin{equation}
u_{i}^{n+1}=\sum_{m\in \Z^d} b_{i,m}(a^n_m,\;u_m^n),\label{eq:sheme0}
\end{equation}
where the $b_{i,m}$ are functions normalized so that $b(a,0)=0$. The non-linear dependence on the velocity field $a$ can for instance follow from upwinding. The velocity field itself is discretized so that for any $n$ and $i$, $a^n_m$ is a vector of $\R^d$~: $a^n_m=(a^n_{m,1},\ldots,a^n_{m,d})$.

We do not explicitly distinguish the boundary conditions in the scheme. And we a priori allow to work on an unbounded grid, hence the fact that $m$ is summed over all $\Z^d$ in \eqref{eq:sheme0}. But of course in most practical settings the grid is truncated, meaning that $a^n_i$ and $u^n_i$ have compact support in $i$. 

We also point out that the function $b_{i,m}$ could be chosen differently from one time step to another, {\em i.e.} $b_{i,m}=b^n_{i,m}$, without any difference in the proofs. In order to avoid additional indices in the notation as much as possible, we  will just write $b_{i,m}$.  

The {\em key assumption} on the $b_{i,m}$ is that $b_{i,m}(a,u)$ is increasing in $u$. This ensures that the scheme is monotone and entropic. For simplicity we also assume that $b_{i,i}(a,u)-u/2$ is increasing in $u$, which can always be insured by choosing the appropriate CFL condition.  

Most explicit schemes require a CFL condition to be monotone and this in turns typically demands a uniform bound on the velocity field
\[
\sup_{i,n} |a_i^n|<\infty.
\]
We do not use directly such a bound but again it is likely to be indirectly imposed through the previous monotonicity condition.

We are asking that the scheme be conservative, which means that $b_{i,m}$ can be expressed as a difference of two fluxes $F^k$ in each direction $k=1\dots d$
\begin{equation}
b_{i,m}(a_m^n,u_m^n)=u_m^n\,\delta_{i=m}+\frac{\delta t}{\delta x}\,\sum_{k=1}^d \left(F^k_{i+[1]_k-m}(a_m^n,u_m^n)-F^k_{i-m}(a_m^n,u_m^n)\right).\label{conservative}
\end{equation}
The fluxes $F_j$ are obviously defined up to a constant function and we normalize them so that for any $a\in \R^d$, $u\in\R_+$
\begin{equation}
\sum_j F^k_j(a,u)=a^k\,f(u),\label{normflux}
\end{equation}
where $a^k$ is the $k$ coordinate of the vector $a$. 

The conservative form of the scheme implies for instance the conservation of mass
\begin{equation}
\sum_{i,\,m \in \Z^d} b_{i,m}(a^n_m,\;u_m^n)=\sum_{i\in\Z^d} u_i^n.\label{conservemass}
\end{equation}
In general we do not ask that the divergence of $a$ be exactly discretized by the scheme. Instead we impose the following condition: There exists a Lipschitz function $\tilde f$ and uniformly bounded $D^n_i$ s.t. for any constant $U$
\begin{equation}
\sum_{j\in \Z^d} b_{i,j}(a^n_j,U)=U+\delta t\,D^n_i\,\tilde f(U).\label{divcondition}
\end{equation}
The uniform bound on the $D^n_i$ obviously correspond to the bound $\|\mbox{div}\,a\|_{L^\infty}$. We also impose that
\begin{equation}
\|\tilde f\|_{W^{1,\infty}}\leq C\,\|f\|_{W^{1,\infty}}. \label{boundftilde}
\end{equation}

This general expression of the discretization allows for many (but not all) of the classical schemes such as Lax-Friedrichs, upwind schemes...
In particular multi-points schemes are included in the formulation. We only impose that too much weight not be given to far away points. This translates into a simple moment condition on the flux~: There exists a constant $C$ and $0<\gamma\leq 1$ s.t.
\begin{equation}
\sum_m |i-m|^\gamma\,|F^k_{i-m}(a_m^n,u_m^n)|\leq C\,\|f'\|_{L^\infty}\,\|a^n\|_{l^p}\,\|u^n\|_{l^{p^*}}. \label{momentsflux}
\end{equation}
We now introduce the discretized version of our semi-norms, namely
\begin{equation}
\|u^n\|_{\alpha,p,\theta}^p=\sup_{h\leq \delta x^{\alpha}} |\log h|^{\theta}\,\delta x^{2d}\,\sum_{i,j} K^h_{i-j}\,|u_i^n-u_j^n|,\label{discreteseminorms}
\end{equation}
where the discrete kernel is directly obtained from $K_h(x)$ given by \eqref{defKh} through
\[
K^h_{i}=K_h(i\,\delta x).
\]
The main difference with respect to the continuous semi-norms if that we do not take the supremum over all possible values of $h$. This is because they do not make sense below a certain size depending on the discretization length $\delta x$. This is why we limit the supremum to those $h\geq \delta x^{\alpha}$. Those semi-norms still provide compactness whenever $\alpha>0$ and $\theta<1$.
\subsection{Main Result}
We are then able to obtain the exact equivalent of Theorem \ref{thregularity}
\begin{thm} Assume that $u^n_i$ is a solution to the recursive scheme given by \eqref{eq:sheme0}-\eqref{conservative} with functions $b_{i,j}(a,u)$ increasing in $u$ and s.t. $b_{i,i}(a,u)-u/2$ is increasing in $u$. Assume moreover that the scheme satisfies \eqref{normflux}, \eqref{conservemass} together with the discretization of the divergence provided that \eqref{divcondition} with the bound \eqref{boundftilde}. Assume finally that the moments' condition \eqref{momentsflux} is satisfied for some $0<\gamma\leq 1$. Then for any $\alpha>0$, any $1<p\leq 2$, any $\theta\geq 1-1/p$ and any $q>p^*$, one has the bound
\[\begin{split}
\|u^n\|_{\alpha,1,\theta}\leq &e^{C\,\|f'\|_{L^{\infty}}\,\sup_m\|D^m\|_{L^\infty}\,n\,\delta t}\,\left(
\|f'\|_{L^{\infty}}\,\sup_m \|u^m\|_{l^{q}} \,\delta t\,\sum_{m\leq n} \|a^m\|_{d,W^{1,p}}
\right.\\
&+\|f'\|_{L^{\infty}}\,\sup_m \|u^m\|_{l^{p^*}}\,\delta t\,\sum_{m\leq n} \|D^m\|_{\alpha,p,p(\theta-1/p^*)}+\|u^0\|_{\alpha,1,\theta}\\
&\left.+\delta x^{\gamma-\alpha-\alpha\,\theta}\,\|f'\|_{L^{\infty}}\,\sup_m \|u^m\|_{l^{p^*}}\,\delta t\,\sum_{m\leq n} \|a^m\|_{l^p}\right).
\end{split}
\]
\label{convergencenumerics}
\end{thm}
\begin{rem}
One can always choose $\theta=1/2$ in the previous estimate. However one has to choose $\alpha$ in terms of $\gamma$. This is always possible, {\em i.e.} there always exists $\alpha>0$ s.t. $\gamma-\alpha-\alpha\,\theta>0$. This ensures that the last term in $O(\delta x^{\gamma-\alpha-\alpha\,\theta})$ indeed vanishes as the grid size decreases to $0$.
\end{rem}
This result easily implies the compactness of the solutions to the scheme with for example
\begin{corollary} (Convergence of the scheme)
Consider a sequence of solutions $u^{n,m}_i$ for schemes satisfying the assumptions of Theorem \ref{convergencenumerics} with a grid size $\delta x_m\rightarrow 0$ and a time step $\delta t_m\rightarrow 0$. Assume that the following quantities are uniformly in $m$, namely for some $q>p^*$ and some $\theta>0$ and $\alpha>0$
\[
\begin{split}
& \sup_m\,\delta t_m\,\sum_{n\leq T/\delta t_m} \|a^{n,m}\|_{d,W^{1,p}}<\infty,\quad \sup_m\,\sup_n \|u^{n,m}\|_{l^{q}}<\infty,\\
&\sup_m\,\sup_{n,i} |D^{n,m}_i|<\infty,  \quad \sup_m\,\sup_n \|D^{n,m}\|_{\alpha,p,\theta}<\infty,\quad \sup_m\,\sup_n \|u^{0,m}\|_{\alpha,1,\theta}<\infty.
\end{split}
\]
Consider the sequence of functions $\tilde u^m$ on $[0,\ T]\times \R^d$ piecewise constant and equal to $u^{n,m}_i$ on the time interval $[n\,\delta t_m,\;(n+1)\,\delta t_m)$ times the cube centered at $x_i=i\,\delta x_m$ and of size $\delta x_m$. Then the sequence $\tilde u^m$ is compact in $L^1_{loc}$. 
\end{corollary}
\begin{rem}
With a minor variation, the assumption that
\[
\sup_m\,\sup_n \|D^{n,m}\|_{\alpha,p,\theta}<\infty,\quad \sup_m\,\sup_n \|u^0\|_{\alpha,1,\theta}<\infty
\]
can be replaced by assuming appropriate compactness on the initial data $u^{0,m}$ and the discrete divergence $D^{n,m}_i$.
\end{rem}
%
\section{The main commutator estimate \label{sec:commutator}}
\subsection{The estimate}
The key point in the proof of the results of this article is a commutator estimate, quantifying the basic one introduced in \cite{DL}. This estimate is important in itself and is likely to be of further use. 
\begin{prop}
Let $1<p<\infty,$ $\exists C<\infty$ depending only on $p$ and the dimension s.t. $\forall a\in W^{1,p}(\R^d)$ with $1\leq p\leq 2$ and 
$\forall g\in L^{2p^*}$ with $1/p^*=1-1/p$,\label{commutatorestimate}
\begin{align*}
&\int_{\mathbb{R}^{2d}}\nabla K_{h}\left(x-y\right)\left(a\left(x\right)-a\left(y\right)\right)\,
\left|g\left(x\right)-g\left(y\right)\right|^{2}\,dx\,dy\\
&\quad\leq C\left\Vert \nabla a\right\Vert _{B^0_{p,q}}\,\left|\log h\right|^{1-1/q}\,\left\Vert g\right\Vert_{L^{2\,p^*}}^2\\
&\quad\ +C\,\left\Vert \mbox{div}\,a\right\Vert _{L^{\infty}}\,\int_{\mathbb{R}^{2d}}K_{h}\left(x-y\right)\,\left|g\left(x\right) -g\left(y\right)\right|^{2}\,dx\,dy.
\end{align*}
In particular using $q=2$,
\begin{align*}
&\int_{\mathbb{R}^{2d}}\nabla K_{h}\left(x-y\right)\left(a\left(x\right)-a\left(y\right)\right)\,
\left|g\left(x\right)-g\left(y\right)\right|^{2}\,dx\,dy\\
&\quad\leq C\left\Vert \nabla a\right\Vert _{L^p}\,\left|\log h\right|^{1/2}\,\left\Vert g\right\Vert_{L^{2\,p^*}}\\
&\quad\ +C\,\left\Vert \mbox{div}\,a\right\Vert _{L^{\infty}}\,\int_{\mathbb{R}^{2d}}K_{h}\left(x-y\right)\,\left|g\left(x\right) -g\left(y\right)\right|^{2}\,dx\,dy.
\end{align*}
\end{prop}
%
\subsection{The connection with the classical theory of renormalized solutions}
In essence Prop. \ref{commutatorestimate} is a quantified version of the classical commutator estimate at the heart of the theory of renormalized solutions, which for this reason we describe more here.

This theory was introduced by DiPerna and Lions in \cite{DL} to handle the well posedness of weak solutions to the linear continuity equation \eqref{eq:continuity}, which we recall is
\[
\partial_t u+\textrm{div}\left(a\left(t,x\right) \,u\left(t,x\right)\right)=0.
\]
If $a\in L^1_t L^p_x$, a weak solution $u\in L^\infty_t L^q_x$ to \eqref{eq:continuity} is said to be renormalized iff for any $\beta\in C^1(\R)$ with $|\beta(\xi)|\leq C\,|\xi|$, $\beta(u)$ is a solution to
\[
\partial_t \beta(u)+\textrm{div}\left(a\left(t,x\right) \,\beta(u)\right)+\textrm{div}\,a\,(\beta'(u)\,u-\beta(u))=0.
\]
An equation for a given $a$ is renormalized iff all weak solutions are renormalized. This is now an important property which directly implies uniqueness: If $u$ is a weak solution with $u(t=0)=0$ then $|u|$ is also a weak solution and hence
\[
\int |u(t,x)|\,dx=\int |u(t=0,x)|\,dx=0.
\]
This also indirectly implies the compactness of any sequence of solutions $u_n$. Essentially one combines the uniqueness with the renormalization property at the limit to prove that
\[
w-\lim \beta(u_n)=\beta(w-\lim u_n),
\] 
where $w-\lim$ denotes the weak limit in the appropriate $L^q$ space.

DiPerna and Lions proved in \cite{DL} that if $a\in L^1_t W^{1,1}_x$ then any bounded solution $u$ is renormalized. The proof relies on the following simple but powerful idea: Given a weak solution $u$, consider a smooth convolution kernel $K_\eps$ and convolve the equation with $K_\eps$
\[
\partial_t K_\eps\star u+\textrm{div}\left(a \,K_\eps\star _x u\right)=R_\eps.
\]
Of course $K_\eps\star u$ cannot be also a solution and there is a remainder term which can be rewritten as
\begin{equation}\begin{split}
R_\eps&=\int \textrm{div}_x\left((a(x)-a(y))\,K_\eps(x-y)\,u(y)\right)\,dy\\
&=\int (a(x)-a(y))\,\nabla K_\eps(x-y)\,u(y)+\int \textrm{div}_x a(x)\,K_\eps(x-y)\,u(y)\,dy\\
&=C_\eps+D_\eps. 
\end{split}\label{commut}
\end{equation}
For a fixed $a\in L^1_t W^{1,1}_x$ and $u\in L^\infty$, one can then prove that $R_\eps$ converges to $0$ in $L^1$. 

It is then straightforward to write an equation on $\beta(u_\eps)$ and pass to the limit as $\eps\rightarrow 0$.

This idea was then extended to include $a\in L^1 BV_x$, first in \cite{Bo} in the kinetic context and then in the seminal \cite{Am} in the general case; we also refer to \cite{HLL, LL}. Those require the use of specific kernels, based on a quadratic form in $\R^d$ that is adapted to the singular part of $\nabla a$. 
 
In general, without any additional structure, $BV$ seems to be the critical space here as proved in \cite{DeP}. If some additional structure is available, then one may be able to work with less. Typical examples are found in dimension $2$, in \cite{ABC, CCR, CR, Ha1} or with some phase space structure in \cite{CJ,JabMas}.

As one can readily the quantity that we bound explicitly in Prop. \ref{commutatorestimate} is very close to the commutator estimate \eqref{commut}. In fact this proposition could be used to directly give an explicit bound in $\eps$ in \eqref{commut} for the particular $K_\eps$ that we use. It is in this sense that we talk of a quantified commutator estimate.

The fact that renormalized solutions are connected to some regularity of the solutions had been noticed for instance in \cite{ADM}. But this regularity had not been quantified until \cite{CD} that we mentioned earlier. As we explained, \cite{CD} involves a quantitative regularity estimate at the level of the characteristics. The approach that we follow here is very different and it is a nice feature to be able to quantify exactly this commutator. 

For more on renormalized solutions, we refer for example to the surveys in \cite{AmCr, DeL}.
\subsection{Useful Technical Lemmas}
The proof of Prop. \ref{commutatorestimate} requires the use of two classical lemmas. First to relate the difference $a(x)-a(y)$ to $\nabla a$,
\begin{lem}
\label{lem:psi} There exists a bounded function $\psi$ s.t.  $\psi$ is $W^{1,1}$ on $B\left(0,1\right)\times S^{d-1}$ and for any $a\in (BV_{loc})^d$
\[
\begin{split}
 a_{i}\left(x\right)-a_{i}\left(y\right)  &=\left|x-y\right|\int_{B\left(0,1\right)}\psi\left(z,\frac{x-y}{\left|x-y\right|}\right)\cdot\nabla a_{i}\left(x+\left|x-y\right|z\right)\frac{dz}{\left|z\right|^{d-1}}\\
 &  +\left|x-y\right|\int_{B\left(0,1\right)}\psi\left(z,\frac{x-y}{\left|x-y\right|}\right)\cdot\nabla a_{i}\left(y+\left|x-y\right|z\right)\frac{dz}{\left|z\right|^{d-1}}.
\end{split}
\]
Moreover for some given constant $\alpha$
\[
\int_{B\left(0,1\right)}\psi\left(z,\frac{x-y}{\left|x-y\right|}\right)\frac{dz}{\left|z\right|^{d-1}}=\alpha\,\frac{x-y}{\left|x-y\right|}.
\]
\end{lem}
The proof of Lemma \ref{lem:psi} is straightforward; it consists in integrating $\nabla v$ over a curve between $x$ and $y$ and then averaging the resulting estimate over all such curves. We refer for instance to \cite{CJ} for a detailed proof.

We need another technical result to control slightly ``delocalized'' convolutions of $\nabla a$ 
\begin{lem}
\label{lem:convolution} For any $1<p<\infty,$ any $L\in W^{s,1}$
for some $s>0$ with compact support and $\int L=0,$ there exists
$C>0$s.t. for any $u\in L^{p}\left(\mathbb{R}^{d}\right)$ 
\begin{equation}
\int_{h_{0}}^{1}\left\Vert L_{r}*u\right\Vert _{L^{p}}\frac{dr}{r}\leq C\,|\log h_0|^{1-1/q}\,\left\Vert u\right\Vert _{B_{p,q}^{0}}\label{eq:conv1}
\end{equation}
where $L_{r}\left(x\right)=r^{-d}L\left(\nicefrac{x}{r}\right)$ and
the constant $C$ depends only on $W^{s,1}$and the size of the support
of $L.$ As a consequence for $p\leq2$
\begin{equation}
\int_{h_{0}}^{1}\left\Vert L_{r}*u\right\Vert _{L^{p}}\frac{dr}{r}\leq C\left|\log h_{0}\right|^{\nicefrac{1}{2}}\left\Vert u\right\Vert _{L^{p}}.\label{eq:conv2}
\end{equation}
\end{lem}
The proof is again classical and is given in the appendix for the sake of completeness. 
\subsection{The proof of Prop. \ref{commutatorestimate}}
%
Observe that by the definition of $K_h$, 
\[
\nabla K_{h}=\frac{x-y}{\left(\left|x-y\right|+h\right)^{d+1}}\,\chi(|x-y|),
\]
for some smooth function $\chi$ with support in $|x-y|\leq 2$ with $\chi=1$ if $|x-y|\leq 1$.

Using Lemma \ref{lem:psi}, one obtains
\begin{align*}
&\int_{\mathbb{R}^{2d}}\nabla K_{h}\left(x-y\right)\left(a\left(x\right)-a\left(y\right)\right)\left|g\left(x\right)-g\left(y\right)\right|^2\,dx\,dy\\
&\quad=2\,\int\frac{(x-y)\,\chi}{\left(\left|x-y\right|+h\right)^{d+1}}\,\int_{B\left(0,1\right)}\psi\left(z,\frac{x-y}{\left|x-y\right|}\right).\nabla a\left(x+\left|x-y\right|z\right)\frac{dz}{\left|z\right|^{d-1}}\\
&\qquad\qquad\left|g\left(x\right)-g\left(y\right)\right|^{2}dxdy,
\end{align*}
where by the symmetry of the expression in $x$ and $y$, both terms in Lemma \ref{lem:psi} lead to the same expression.

We introduce the average of $\psi$ as given by Lemma \ref{lem:psi} and decompose accordingly
\begin{align*}
&\int\frac{(x-y)\,\chi}{\left(\left|x-y\right|+h\right)^{d+1}}\,\int_{B\left(0,1\right)}\psi\left(z,\frac{x-y}{\left|x-y\right|}\right)\cdot \nabla a\left(x+\left|x-y\right|z\right)\,\frac{dz}{\left|z\right|^{d-1}}\\
&\qquad\qquad\qquad\left|g\left(x\right)-g\left(y\right)\right|^{2}\,dx\,dy=I+\bar\alpha\,J,
\end{align*}
where
\begin{align*}
I= &  \int\frac{(x-y)\,\chi}{\left(\left|x-y\right|+h\right)^{d+1}}\cdot\int_{B\left(0,1\right)}\left(\psi\left(z,\frac{x-y}{\left|x-y\right|}\right)-\right.\\
 & \left. \bar\alpha\,\frac{x-y}{\left|x-y\right|}\right)\cdot\nabla a\left(x+\left|x-y\right|z\right)\,\frac{dz}{\left|z\right|^{d-1}}\, \left|g\left(x\right)-g\left(y\right)\right|^{2}\,dx\,dy,
\end{align*}
 and 
\begin{align*}
J= & \int\frac{\left(x-y\right)_{i}\,\chi}{\left(\left|x-y\right|+h\right)^{d+1}}\,\frac{\left(x-y\right)_{j}}{\left|x-y\right|}\,\int_{B\left(0,1\right)}\partial_{i}a_{j}\left(x+\left|x-y\right|z\right)\,\frac{dz}{\left|z\right|^{d-1}}\\
&\qquad \left|g\left(x\right)-g\left(y\right)\right|^{2}\,dx\,dy,
\end{align*}
where we used Einstein convention of summation.

The constant $\bar \alpha$ is chosen so that
\[
\int\left(\psi\left(z,\frac{x-y}{\left|x-y\right|}\right)-C\frac{x-y}{\left|x-y\right|}\right)\frac{dz}{\left|z\right|^{d-1}}=0,
\]
which is always possible thanks to Lemma \ref{lem:psi}. Both terms $I$ and $J$ rely on appropriate cancellations that allow to use Lemma \ref{lem:convolution} but on different terms. As such we have to handle them separately.

\bigskip

\noindent \textbf{Control of $I.$} Denote for simplicity
\[
\tilde{\psi}\left(z,\omega\right):=\psi\left(z,\omega\right)-
\bar\alpha\,\omega,\quad
L\left(z,\omega\right):=\tilde{\psi}\left(z,\omega\right)\frac{\mathbb{I}_{|z|\leq 1}}{\left|z\right|^{d-1}}.
\]
One can see easily that for any fixed $\omega$, $L$ is compactly supported and belongs to $W^{s,1}$ for any $s<1$. It hence satisfies the assumptions of Lemma \ref{lem:convolution}, uniformly in $\omega$.

Now observe that
\[
\int_{B\left(0,1\right)}\frac{\tilde{\psi}\left(z,w\right)}{\left|z\right|^{d-1}}\cdot\nabla a\left(x+rz\right)\,dz=L_r(.,\omega)\star \nabla a,
\]
where $L_r(z,\omega)=r^{-d}\,L(z/r,\omega)$. On the other hand by the spherical changes of variables
\[
\begin{split} I=&\int_{0}^{1}\int_{S^{d-1}}\int\frac{r^{d}\,\chi(r)}{\left(r+h\right)^{d+1}}\int_{B\left(0,1\right)}\frac{\tilde{\psi}\left(z,\omega\right)}{\left|z\right|^{d-1}}\nabla a\left(x+rz\right)\,dz\\
&\qquad\qquad\qquad\left|g\left(x\right)-g\left(x-r\omega\right)\right|^{2} \,dx\,d\omega\,dr\\
\leq &\int_{S^{d-1}}\int_0^1\int \frac{r^{d}}{\left(r+h\right)^{d+1}}\,|L_r(.,\omega)\star \nabla a|\,\left|g\left(x\right)-g\left(x-r\omega\right)\right|^{2}.\end{split}
\]
By a H\"older estimate
\[
\begin{split}
I\leq&\int_0^1\int_{S^{d-1}}\frac{1}{r+h}\,\left\|L_{r}(.,\omega)\star\nabla a\right\|_{L^p}\, \left\|g\left(.\right)-g\left(.-rw\right)\right\|_{L^{2\,p^*}}\,d\omega\,dr.
\end{split}
\]
Of course
\[
\begin{split}
&\left\|g\left(.\right)-g\left(.-rw\right)\right\|_{L^{2\,p^*}}\leq 2\,\left\|g\left(.\right)\right\|_{L^{2\,p^*}}.
\end{split}
\]
We recall here that $L$ satisfies the assumption of Lemma \ref{lem:convolution} uniformly in $\omega$. We hence deduce for some constant $C$
\[
I\leq C\,\left\Vert \nabla a\right\Vert _{B^0_{p,q}}\,\left|\log h\right|^{1-1/q}\left\Vert g\right\Vert _{L^{2\,p^*}}.
\]

\noindent \textbf{Control of J.} The general idea is similar to $I$: Trying to identify convolution with a kernel of vanishing average to gain regularity.

For this reason, we decompose again
\[
J=J_1+J_2,
\]
where we subtracted the right average in $J_1$
\[\begin{split}
J_1=   \int&\frac{\left|x-y\right|\,\chi}{\left(\left|x-y\right|+h\right)^{d+1}}\left(\frac{\left(x-y\right)}{\left|x-y\right|}\otimes\frac{\left(x-y\right)}{\left|x-y\right|}-\tilde{C}I_{d}\right)\\
&\qquad:\nabla a\left(x+\left|x-y\right|z\right)\frac{dz}{\left|z\right|^{d-1}}\, \left|g\left(x\right)-g\left(y\right)\right|^{2}\,dx\,dy,
\end{split}
\]
where $A:B$ denotes the total contraction of two matrices ${\displaystyle \sum_{i,j}A_{ij}B_{ij}}$ and where $\tilde C$ is again chosen s.t.
\[
\int_{S^{d-1}} (\omega_i^2-\tilde C)\,d\omega=0.
\]
This leaves as $J_2$
\[\begin{split}
J_2=\tilde{C}\int&\frac{\left|x-y\right|\,\chi}{\left(\left|x-y\right|+h\right)^{d+1}} \,I_{d}\,:\,\nabla a\left(x+\left|x-y\right|z\right)\frac{dz}{\left|z\right|^{d-1}}\\
&\quad\left|g\left(x\right)-g\left(y\right)\right|^{2}\,dx\,dy.
\end{split}
\]
This term can be immediately bounded as
\[
I_d:\nabla a=\mbox{div}\,a,
\]
so that
\begin{equation}\begin{split}
J_2&\leq C\,\|\mbox{div}\,a\|_{L^\infty}\int_{\Pi^{2d}} \frac{|g(x)-g(y)|^2}{(|x-y|+h)^d}\,dx\,dy\\
&\leq C\,\|\mbox{div}\,a\|_{L^\infty}\int_{\Pi^{2d}} (1+K_h(x-y))\,|g(x)-g(y)|^2\,dx\,dy.\label{boundJ2}
\end{split}
\end{equation}
We now turn to $J_1$ where we need to use a slight variant of spherical coordinates by writing  $x-y=-r\,w$ for $r\in \R$ and $1/4\leq w\leq 1$ instead of $w\in S^{d-1}$ with $|w|=1$ as usual. Indeed for a fixed $x\in \R^d$ use first spherical coordinates, $w=s\,\omega$ to calculate
\[\begin{split}
&\int_{0}^{R} \int_{1/2\leq|w|\leq R} \Phi(x+r\,w)\,r^{d-1}\,dw\,dr\\
&\quad=\int_{1/2}^{1} \int_0^R\int_{S^{d-1}} \Phi(x+s\,r\,\omega)\,(r\,s)^{d-1}\,dr\,d\omega\,ds\\
&\quad=\int_{1/2}^{1} \int_0^{R\,s}\int_{S^{d-1}} \Phi(x+\tilde r\,\omega)\,\tilde r^{d-1}\,d\tilde r\,d\omega\,\frac{ds}{s}=\int_{1/2}^{1} \int_{|x-y|\leq R\,s} \Phi(y)\,dy\,\frac{ds}{s}\\
&\quad=\int_{|x-y|\leq R} \Phi(y)\int_{\max(1/2,|x-y|/R)}^{1}\frac{ds}{s}\,dy,
\end{split}
\]
with the change of variables $r\rightarrow \tilde r=r\,s$ for a fixed $s$ and a final use of spherical coordinates $y=x+\tilde r\,\omega$. Therefore defining the smooth function $W_R(r)=\int_{\max(1/2,r/R)}^1 s^{-1}\,ds$ and for any $\tilde \Phi$ by taking $\Phi(y)=\tilde\Phi(y)/W_R(|x-y|)$
\begin{equation}
\int_{|x-y|\leq R} \tilde\Phi(y)\,dy=\int_{0}^{R} \int_{1/2\leq|w|\leq R} \tilde \Phi(x+r\,w)\,\frac{r^{d-1}}{W_R(r\,|w|)}dw\,dr.\label{sphericalmod}
\end{equation}
Denote accordingly
\[
\tilde L\left(w\right)=\left(\frac{w\otimes w}{\left|w\right|^{2}}-\tilde{C}I_{d}\right)\,\mathbb{I}_{1/2\leq w\leq 1}.
\]
This allows to rewrite
\[\begin{split}
J_{1}= &  \int_0^1\int\frac{r^d\,\chi(r)}{W_1(r\,|w|)\,\left(r+h\right)^{d+1}}\,\tilde L\left(w\right)\,:\, \int_{B(0,1)}\nabla a\left(x+r\,z\right)\frac{dz}{\left|z\right|^{d-1}}\\
&\qquad\left|g\left(x\right)-g\left(x+rw\right)\right|^{2}\,dx\,dw\,dr.
\end{split}
\]
Observe that $W_1(u)=\bar w$ is constant for $u<1/2$, that is since $|w|\leq 1$ that $W_1(r\,|w|)=\bar w$ if $r<1/2$. Obviously the integral $J_1$ is bounded for $r>1/2$ so
\[\begin{split}
J_{1}\leq &C+  \frac{1}{\bar w}\int_0^1\int\frac{r^d\,\chi(r)}{\left(r+h\right)^{d+1}}\,\tilde L\left(w\right)\,:\, \int_{B(0,1)}\nabla a\left(x+r\,z\right)\frac{dz}{\left|z\right|^{d-1}}\\
&\qquad\left|g\left(x\right)-g\left(x+rw\right)\right|^{2}\,dx\,dw\,dr.
\end{split}
\]
Now denote for simplicity
\[
A_r(x)=\int_{B(0,1)}\nabla a\left(x+r\,z\right)\frac{dz}{\left|z\right|^{d-1}},
\]
and let us expand $\left|g\left(x\right)-g\left(y\right)\right|^{2}=g^2(x)+g^2(y)+2g(x)\,g(y)$ so that 
\[
J_1\leq C+\frac{J_x+2\,J_{xy}+J_y}{\bar w},
\]
with
\[
\begin{split}
&J_x=\int_0^1\int\frac{r\,\chi(r)}{\left(r+h\right)^{d+1}}\,\tilde L\left(w\right)\, :\,A_r(x)\,g^2(x)\,dx\,dw\,dr,\\
&J_{xy}=\int_0^1\int\frac{r\,\chi(r)}{\left(r+h\right)^{d+1}}\,\tilde L\left(w\right)\, :\,A_r(x)\,g(x)\,g(x+rw)\,dx\,dw\,dr,\\
&J_x=\int_0^1\int\frac{r\,\chi(r)}{\left(r+h\right)^{d+1}}\,\tilde L\left(w\right)\, :\,A_r(x)\,g^2(x+rw)\,dx\,dw\,dr.\\
\end{split}
\]
First note that since $\int \tilde L(w)\,dw=0$, one simply has that
\[
J_x=0.
\]
The term $J_y$ can be controlled through Lemma \ref{lem:convolution} as one can identify a convolution
\[\begin{split}
J_y&=\int_{\R^{d}}\int_{0}^{1}\frac{r^{d}\,\chi(r)}{\left(r+h\right)^{d+1}}\,A_r(x)\,:\,\tilde L_{r}\star\left|g\right|^{2}\,dx\,dr\\
&=\int_{\R^{d}}\int_{0}^{1}\frac{r^{d}\,\chi(r)}{\left(r+h\right)^{d+1}}\,\tilde L_r \star A_r(x)\,\left|g\right|^{2}\,dx\,dr,
\end{split}\]
since $\tilde L$ is even. Now, by
H\"older inequality one has
that
\[\begin{split}
J_{y}&\leq \int_0^1 \frac{1}{r+h}\,\|\tilde L_r\star A_r\|_{L^{p}}\,\|g^2\|_{L^{p^*}}\,dr.\\
\end{split}\]
Of course denoting $B(z)=\mathbb{I}_{|z|\leq 1}\,|z|^{1-d}$, one has by the definition of $A_r$, 
\[
\|\tilde L_r\star A_r\|_{L^p}=\|\tilde L_r\star(B_r\star \nabla a)\|_{L^p}=\|B_r\star(\tilde L_r\star \nabla a)\|_{L^p} \leq\left\|B_r\right\|_{L^1}\leq \|\tilde L_r\star\nabla a\|_{L^p}.
\]
Finally, by Lemma 3 applied to the exponent $p$, we deduce that 
\[
J_{y}\leq C\,\left\Vert g^{2}\left(t,.\right)\right\Vert _{L^{p^*}}\,\|\nabla a\|_{B^0_{p,q}}\,\left|\log h\right|^{1-1/q}.
\]
Again we identify the convolution in $J_{xy}$ 
\[
\begin{split}
J_{xy}= &  \int_{\Pi^{d}}\int_{0}^{1}\frac{r^d\,\chi(r)}{(r+h)^{d+1}}\,L_{r}\star g(x)\,:\,A_r(x)\,g\left(t,x\right)dxdr,
\end{split}
\]
and still by Lemma \ref{lem:convolution}
we get 
\[
\begin{array}{c}
J_{xy}\leq C\,\left\Vert g\left(t,.\right)\right\Vert_{L^{2p^*}}^2 \,\left|\textrm{log}\,h\right|^{1-1/q}\, \left\Vert\nabla a\right\Vert_{B^0_{p,q}}.\end{array}
\]
Collecting all estimates and recalling that for $p\leq 2$, $L^p\subset B^0_{p,2}$, concludes the proof of the proposition.
%
\section{Proof of Proposition \ref{propembedding}  and Theorem \ref{thregularity}}
%
\subsection{Proof of Prop. \ref{propembedding}}
The embedding $W^p_{\log,\theta}\subset L^p$ is straightforward from the definition. For the embedding $W^{s,p}\subset W^p_{\log,\theta}$, note that if $u\in W^{s,p}$
\[
\|u(.+z)-u(.)\|_{L^p}\leq C\,|z|^{s}\,\|\nabla u\|_{L^p},
\]
which can easily be obtained by interpolation from the $s=1$ case
\[\begin{split}
\|u(.+z)-u(.)\|_{L^p}&=\left\|\int_0^1 z\nabla u(.+sz)ds\right\|_{L^p}\leq \int_0^1 |z|\,\left\|\nabla u(.+sz)\right\|_{L^p}\,ds\\
&=|z|\,\left\|\nabla u\right\|_{L^p}.
\end{split}
\]
On the other hand
\[\begin{split}
\|u\|_{p,\theta}^p=&\sup_h |\log h|^{-\theta}\int_{\R^d} K_h(z) \|u(.+z)-u(.)\|_{L^p}^p\,dz\\
&\leq \|\nabla u\|_{L^p}\,\sup_h |\log h|^{-\theta}\,\int_{\R^d} |z|^s\,K_h(z)\,dz,
\end{split}
\]
while finally
\[
\int_{\R^d} |z|^s\,K_h(z)\,dz\leq \int_{|z|\leq 2} \frac{|z|^s}{|z|^d}\,dz\leq C,
\]
concluding the bound.

Note that with similar calculations
\[
\int_{\R^d} K_h(z)\,dz\sim \int_{|z|\leq 2} \frac{1}{(h+|z|)^d}\,dz\sim |\log h|,
\]
which implies that
\[
\|u\|_{p,1}^p\leq 2\,\|u\|_{L^p}\,\sup_h |\log h|^{-1}\,\int_{\R^d} K_h(z)\,dz\leq C\,\|u\|_{L^p},
\]
so that the semi-norms do not carry any special information when $\theta=1$.

Define now
\[
\bar K_h=\frac{K_h}{\|K_h\|_{L^1}}\sim \frac{K_h}{\log h}.
\]
The kernel $\bar K_h$ is normalized and moreover for any $\delta>0$, as $h\rightarrow 0$,
\[
\int_{|z|\geq \delta} \bar K_h(z)\,dz\longrightarrow 0,
\]
such that $\bar K_h$ is a classical convolution kernel. Observe that
\[\begin{split}
\|u-\bar K_h\star u\|_{L^p}&\leq \|K_h\|_{L^1}^{-1}\,\int_{\R^{d}} K_h(z)\,\|u(.+z)-u(.)\|_{L^p}\,dz\\
&\leq \left( \|K_h\|_{L^1}^{-1}\,\int_{\R^{d}}K_h(z)\,\|u(.+z)-u(.)\|_{L^p}\,dz\right)^{1/p}\\
&\leq C\,|\log h|^{\theta-1}\,\|u\|_{p,\theta}.
\end{split}\]
This proves by Rellich criterion that for any sequence $u_n$ s.t. $\|u_n\|_{p,\theta}+\|u_n\|_{L^p}$ is uniformly bounded, $u_n$ is locally compact in $L^p$.

Finally let us calculate for $p=2$, using Fourier transform
\[
\int_{z\in\R^d} K_h(z)\,\|u(.+z)-u(.)\|_{L^2}^2\,dz=\int_{\R^d} K_h(z)\,\int_{\R^d} \left|e^{iz\cdot\xi}-1\right|^2\,|{\cal F}\,u(\xi)|^2\,d\xi\,dz.
\]
That leads to calculate
\[\begin{split}
\int_{\R^d} K_h(z)\,\left|e^{iz\cdot\xi}-1\right|^2\,dz&\sim \int_{|z\cdot \xi|\geq 1} K_h(z)+ \int_{|z\cdot \xi|\leq 1} K_h(z)\,|z\cdot \xi|^2\\
&\sim \left|\log \left(\frac{1}{|\xi|}+h\right)\right|+1,
\end{split}\]
concluding the proof.
\subsection{Proof of Theorem \ref{thregularity}}
The proof of Th. \ref{thregularity} mostly follows the steps of \cite{BeJa}, the main improvement being the more precise Proposition \ref{commutatorestimate}. 

First of all by Kruzkov's doubling of variables, see \cite{Kru}, any entropy solution $u$ to \eqref{eq:principale} satisfies in the sense of distributions that
\begin{equation}\begin{split}
&\partial_t |u(t,x)-u(t,y)|+\mbox{div}_x (a(t,x)\,F(u(t,x),u(t,y))\\
&+\mbox{div}_y (a(t,y)\,F(u(t,x),u(t,y))+G(u(t,x),u(t,y))\,\mbox{div}_x a(t,x)\\
&+G(u(t,y),u(t,x))\,\mbox{div}_y a(t,y)\leq 0,
\end{split}\label{kruzkov}
\end{equation}
where 
\begin{equation}
\begin{split}
& F(\xi,\zeta)=(f(\xi)-f(\zeta))\,\mbox{sign}\,(\xi-\zeta)=F(\zeta,\xi),\\
&G(\xi,\zeta)=f(\xi)\,\mbox{sign}\,(\xi-\zeta)-F(\xi,\zeta)=\bar G(\xi,\zeta)\frac{1}{2}\,F(\xi,\zeta),\quad\mbox{with}\\
&\bar G(\xi,\zeta)=\frac{f(\xi)+f(\zeta)}{2}\,\mbox{sign}\,(\xi-\zeta)=-\bar G(\zeta,\xi).
\end{split}\label{defFG}
\end{equation}
Note that up to adding a constant in $f$, we may assume that $f(0)=0$ thus normalizing $\bar G$ s.t. $\bar G(0,0)=0$.
 
For any fixed $h$, $K_h(x-y)$ is a smooth, compactly supported function which we may hence use as a test function for \eqref{kruzkov}, giving
\[\begin{split}
&\frac{d}{dt}\int_{\R^{2d}} K_h(x-y)\,|u(t,x)-u(t,y)|\,dx\,dy\\
&\quad\leq\int_{\R^{2d}} \nabla K_h(x-y)\,(a(x)-a(y))\,F(u(x),u(y))\,dx\,dy\\
&\qquad+\int_{\R^{2d}} K_h(x-y)\,\frac{\mbox{div}\,a(x)+\mbox{div}\,a(y)}{2}\,F(u(x),u(y))\,dx\,dy\\
&\qquad+\int_{\R^{2d}} K_h(x-y)\,(\mbox{div}\,a(x)-\mbox{div}\,a(y))\,\bar G(u(x),u(y))\,dx\,dy.\\
\end{split}
\]
Using the bound on the divergence the second term in the r.h.s. can simply be bounded by
\[
\|f'\|_{L^\infty}\,\|\mbox{div}\,a\|_{L^\infty}\,\int_{\R^{2d}} K_h(x-y)\,|u(t,x)-u(t,y)|\,dx\,dy,
\]
since $F(\xi,\zeta)\leq \|f'\|_{L^\infty}\,|\xi-\zeta|$.
while by H\"older estimate the third term in the r.h.s. is bounded by
\[\begin{split}
&\left(\int_{\R^{2d}} K_h(x-y)\,(\mbox{div}\,a(x)-\mbox{div}\,a(y))^p\right)^{1/p}\\
&\qquad\times\left(\int_{\R^{2d}} K_h(x-y)\,|\bar G(u(x),u(y))|^{p^*}\right)^{1/p^*}\\
&\quad\leq \|f'\|_{L^\infty}\,|\log h|^{\theta}\,\|u\|_{L^{p^*}}\,\|\mbox{div}\,a\|_{p,p\,(\theta-1/p^*)},
\end{split}\]
simply by using that $\bar G(\xi,\zeta)\leq \|f'\|_{L^\infty}\,\frac{\xi+\zeta}{2}$ since $f(0)=0$. 

The combination of those two bounds yields that
\begin{equation}
\frac{d}{dt} \|u\|_{1,\theta}\leq \|f'\|_{L^\infty}\,\|\mbox{div}\,a\|_{L^\infty}\,\|u\|_{1,\theta}+\|f'\|_{L^\infty}\, \|u\|_{L^{p^*}}\,\|\mbox{div}\,a\|_{p,p\,(\theta-1/p^*)}+{\cal C},
\end{equation}
where ${\cal C}$ is the commutator
\[
{\cal C}=\sup_{h} |\log h|^{-\theta}\,\int_{\R^{2d}} \nabla K_h(x-y)\,(a(x)-a(y))\,F(u(x),u(y))\,dx\,dy.
\]
Therefore all the difficulty lies in obtaining a explicit quantitative estimate on this commutator. This is where Prop. \ref{commutatorestimate} comes in, leading to improved, more precise results with respect to \cite{BeJa}.

We still need an additional step to put ${\cal C}$ in precisely the right form for  Prop. \ref{commutatorestimate}. As this is going to be used as well for the numerical scheme, we put the corresponding estimate in a lemma
\begin{lem}
Assume that $f\in W^{1,\infty}(\R_+)$, that $u\in L^{p^*,1}$ and that $a\in B^1_{p,q}$ then provided $\theta\geq 1-1/q$
\[\begin{split}
{\cal C}&=\sup_{h} |\log h|^{-\theta}\,\int_{\R^{2d}} \nabla K_h(x-y)\,(a(x)-a(y))\,F(u(x),u(y))\,dx\,dy\\
&\leq C\,\|\mbox{div}\,a\|_{L^\infty}\,\|f'\|_{L^\infty}\, \|u\|_{1,\theta}+C\,\|\nabla a\|_{B^0_{p,q}}\,\|f'\|_{L^\infty}\,\|u\|_{L^{p^*,1}}.
\end{split}
\] \label{lemcommutator}
\end{lem}
\begin{proof}
Just as in \cite{BeJa}, we use the repartition function of $u$
\[
\kappa(t,x,\xi)=\mathbb{I}_{0\leq \xi\leq u(t,x)},
\]
which from the definition of $F$ in \eqref{defFG} implies the simple representation
\[
F(u(x),u(y))=\int_0^\infty f'(\xi)\,|\kappa(x,\xi)-\kappa(y,\xi)|^2\,d\xi.
\]
This lets us simply rewrite
\[\begin{split}
&{\cal C}=\sup_{h} |\log h|^{-\theta} \int_0^\infty f'(\xi)\\
&\qquad\int_{\R^{2d}} \nabla K_h(x-y)\,(a(x)-a(y))\,|\kappa(t,x,\xi)-\kappa(t,y,\xi)|^2\,dx\,dy\,d\xi.
\end{split}
\]
We may now directly use Prop. \ref{commutatorestimate} to find that provided $\theta\geq 1-1/q$
\[
{\cal C}\leq C\,\|\mbox{div}\,a\|_{L^\infty}\,\|f'\|_{L^\infty}\, \|u\|_{1,\theta}+C\,\|\nabla a\|_{B^0_{p,q}}\,\int_0^\infty f'(\xi)\,\|\kappa(t,.,\xi)\|_{L^{2p^*}}^2\,d\xi.
\]
It is now straightforward to estimate
\[\begin{split}
\int_0^\infty f'(\xi)\,\|\kappa(t,.,\xi)\|_{L^{2p^*}}^2\,d\xi&\leq \|f'\|_{L^\infty}\,\int_0^\infty |\{ u(t,.)\geq \xi\}|^{1/p^*}\,d\,\xi\\
&=\|f'\|_{L^\infty}\,\|u\|_{L^{p^*,1}},
\end{split}\]
where $L^{p^*,1}$ denotes the corresponding Lorentz space. 
\end{proof}

From the previous Lemma we finally obtain that if $\theta\geq \max(1/p^*,1-1/q)$
\[\begin{split}
\frac{d}{dt} \|u\|_{1,\theta}\leq &C\,\|f'\|_{L^\infty}\,\|\mbox{div}\,a\|_{L^\infty}\,\|u\|_{1,\theta}+\|f'\|_{L^\infty}\, \|u\|_{L^{p^*}}\,\|\mbox{div}\,a\|_{p,p\,(\theta-1/p^*)}\\
&+C\,\|\nabla a\|_{B^0_{p,q}}\,\|f'\|_{L^\infty}\,\|u\|_{L^{p^*,1}}.
\end{split}
\]
A Gronwall estimate concludes the first statement of Th. \ref{thregularity}. The embeddings $L^p\subset B^0_{p,2}$ for $p\leq 2$ and $L^r\cap L^1\subset L^{p^*,1}$ for $r>p^*$ conclude the second statement.
\section{The Numerical Scheme: Proof of Theorem \ref{convergencenumerics}}
We have to calculate, denoting $s_{i,j}^{n+1}=\mbox{sign}(u_i^{n+1}-u_j^{n+1})$
\[\begin{split}
&\sum_{i,j\in \Z^d} |u_i^{n+1}-u_j^{n+1}|\,K^h_{i-j}=\sum_{i,j}s_{i,j}^{n+1}\,K^h_{i-j}\,\sum_{m\in \Z^d}\left( b_{i,m}(a^n_m,\;u_m^n)-b_{j,m}(a^n_m,\;u_m^n)\right)\\
&\quad=\sum_{i,j}s_{i,j}^{n+1}\,K^h_{i-j}\,\sum_{m\in \Z^d}\left( b_{i,m}(a^n_m,\;u_m^n)-b_{i,m}(a_m^n,u_j^n)+\frac{1}{2}(u_j^n-u_i^n)\,\delta_{m-i}\right.\\
&\hspace{100pt} \left. +b_{j,m}(a_m^n,u_{i}^n)-b_{j,m}(a^n_m,\;u_m^n)+\frac{1}{2}(u_j^n-u_i^n)\,\delta_{m-j}  \right)\\
&\quad\ +\sum_{i,j,m}s_{i,j}^{n+1}\,K^h_{i-j} \left(b_{i,m}(a_m^n,u_{j}^n)-b_{j,m}(a_m^n,u_{i}^n)-u_j^n+u_i^n\right).
\end{split}\]
Use the discretized expression of the divergence of $a$ given by Eq. \eqref{divcondition} to find that
\[
\sum_m b_{i,m}(a_j^n,u_j^n)=u_j^n+\delta t\,D_j^n\,\tilde f(u_j^n),\quad \sum_m b_{j,m}(a_j^n,u_i^n)=u_i^n+\delta t\,D_i^n\,\tilde f(u_i^n).
\]
We also recall that the $b_i$ is increasing in $u$ s.t. $b_{i,m}(a^n_m,\;u_m^n)-b_{i,m}(a_m^n,u_j^n)$ has the sign of $u_m^n-u_j^n$. In particular
\[
s_{ij}^{n+1}\,(b_{i,m}(a^n_m,\;u_m^n)-b_{i,m}(a_m^n,u_j^n))\leq s_{mj}^{n}\,(b_{i,m}(a^n_m,\;u_m^n)-b_{i,m}(a_m^n,u_j^n)).
\]
Similarly since $b_{i,j}(a,u)-u/2$ is increasing
\[\begin{split}
&s_{ij}^{n+1}\,\left(b_{i,i}(a^n_i,\;u_i^n)-b_{i,m}(a_i^n,u_j^n)+\frac{1}{2}(u_j^n-u_i^n)\right)\\
&\qquad\leq s_{ij}^{n}\,\left(b_{i,i}(a^n_i,\;u_i^n)-b_{i,m}(a_i^n,u_j^n)+\frac{1}{2}(u_j^n-u_i^n)\right). 
\end{split}\]
We hence have
\[\begin{split}
&\sum_{i,j\in \Z^d} |u_i^{n+1}-u_j^{n+1}|\,K^h_{i-j}\\
&\quad\leq \sum_{i,j,m}s_{m,j}^{n}\,K^h_{i-j}\,\left( b_{i,m}(a^n_m,\;u_m^n)-b_{i,m}(a_m^n,u_j^n)+\frac{1}{2}(u_j^n-u_i^n)\,\delta_{m-i}\right)\\
&\qquad+\sum_{i,j,m}s_{i,m}^{n}\,K^h_{i-j}\,\left( b_{j,m}(a_m^n,u_{i}^n)-b_{j,m}(a^n_m,\;u_m^n)+\frac{1}{2}(u_j^n-u_i^n)\,\delta_{m-j}  \right)\\
&\quad\ +\delta t\,\sum_{i,j}s_{i,j}^{n+1}\,K^h_{i-j} \left(D_j^n\,\tilde f(u_j^n)-D_i^n\,\tilde f(u_i^n)\right).
\end{split}
\]
Using now the conservative form of the scheme as given by \eqref{conservative}, we may write that
\begin{equation}
\begin{split}
&\sum_{i,j\in \Z^d} |u_i^{n+1}-u_j^{n+1}|\,K^h_{i-j}\leq \sum_{i,j}\,K^h_{i-j}\, |u_i^n-u_j^n|\\
&\qquad+\frac{\delta t}{\delta x}\,\sum_{i,j,m }\sum_{k=1}^d s_{m,j}^{n}\,K^h_{i-j}\,\left( F^k_{i+[1]_k-m}(a^n_m,\;u_m^n)-F^k_{i+[1]_k-m}(a^n_m,\;u_j^n)\right.\\ &\hspace{100pt}\left. -F^k_{i-m}(a^n_m,\;u_m^n)+F^k_{i-m}(a^n_m,\;u_j^n)\right)\\
&\qquad+\frac{\delta t}{\delta x}\,\sum_{i,j,m }\sum_{k=1}^d s_{i,m}^{n}\,K^h_{i-j}\,\left(F^k_{j+[1]_k-m}(a^n_m,\;u_i^n) -F^k_{j+[1]_k-m}(a^n_m,\;u_m^n)\right.\\
&\hspace{100pt} \left. -F^k_{j-m}(a^n_m,\;u_i^n)+F^k_{j-m}(a^n_m,\;u_m^n)
 \right)\\
&\quad\ +\delta t\,\sum_{i,j}s_{i,j}^{n+1}\,K^h_{i-j} \left(D_j^n\,\tilde f(u_j^n)-D_i^n\,\tilde f(u_i^n)\right),
\end{split}\label{kruzkovdiscrete}
\end{equation}
which is the equivalent of Eq. \eqref{kruzkov}.

Denote
\[
{\cal D}_h=\delta t\,\sum_{i,j}s_{i,j}^{n+1}\,K^h_{i-j} \left(D_j^n\,\tilde f(u_j^n)-D_i^n\,\tilde f(u_i^n)\right).
\]
Following the proof of Theorem \ref{thregularity}, this term can be simply handled
\[
{\cal D}_h\leq \delta t \,\sum_{i,j} K^h_{i-j} |D_i^n-D_j^n|\,\tilde f(u_i^n)+\|D_i^n\|_{l^\infty}\,\delta t \,\sum_{i,j} K^h_{i-j}\,|\tilde f(u_i^n)-\tilde f(u_j^n)|.
\]
By simple Lipschitz bounds and discrete H\"older estimates, one obtains that
\[\begin{split}
{\cal D}_h\leq \delta t\,\|\tilde f\|_{L^{\infty}}\, \delta x^{-2d}&\left(\|u\|_{l^{p^*}}\,\left(\delta x^{2d}\,\sum_{i,j} K^h_{i-j} |D_i^n-D_j^n|^p\right)^{1/p}\right.\\
&+\left.\|D_i^n\|_{l^\infty}\,\delta x^{2d}\,\sum_{i,j} K^h_{i-j}\,|u_i^n-u_j^n|\right).
\end{split}
\]
Since the Lipschitz norm of $\tilde f$ is bounded by the norm of $f$ and by the definition of our discrete semi-norms, we hence obtain the exact equivalent of the continuous case namely for any $h\geq \delta x^{\alpha}$
\begin{equation}
\delta x^{2d}\,|\log h|^{-\theta}\,{\cal D}_h\leq \delta t\,\|f\|_{L^{\infty}}\, \left(\|u\|_{l^{p^*}}\,\|D^n\|_{\alpha,p, p\,(\theta -1/p^*)}+\|D_i^n\|_{l^\infty}\,\|u^n\|_{\alpha,p,\theta}\right).\label{boundDterm}
\end{equation}
We now perform a discrete integration by parts in the other terms of Eq. \eqref{kruzkovdiscrete} with for example
\[\begin{split}
&\sum_{i,j,m }\sum_{k=1}^d s_{m,j}^{n}\,K^h_{i-j}\,\left( F^k_{i+[1]_k-m}(a^n_m,\;u_m^n)-F^k_{i+[1]_k-m}(a^n_m,\;u_j^n)\right.\\ &\hspace{100pt}\left. -F^k_{i-m}(a^n_m,\;u_m^n)+F^k_{i-m}(a^n_m,\;u_j^n)\right)\\
&\ =\sum_{i,j,m }\sum_{k=1}^d s_{m,j}^{n}\,(K^h_{i-[1]_k-j}-K^h_{i-j})\,\left(F^k_{i-m}(a^n_m,\;u_m^n) -F^k_{i-m}(a^n_m,\;u_j^n) \right)
\end{split}
\]
This leads to
\[
\begin{split}
&\sum_{i,j\in \Z^d} |u_i^{n+1}-u_j^{n+1}|\,K^h_{i-j}\leq \sum_{i,j}\,K^h_{i-j}\, |u_i^n-u_j^n|+{\cal D}_h\\
&\quad+\frac{\delta t}{\delta x}\sum_{i,j,m }\sum_{k=1}^d s_{m,j}^{n}\,(K^h_{i-[1]_k-j}-K^h_{i-j})\,\left(F^k_{i-m}(a^n_m,\;u_m^n) -F^k_{i-m}(a^n_m,\;u_j^n) \right)\\
&\quad+\frac{\delta t}{\delta x}\sum_{i,j,m} \sum_{k=1}^d s_{i,m}^{n}\,(K^h_{i+[1]_k-j}-K^h_{i-j})\,\left(F^k_{j-m}(a^n_m,\;u_i^n) -F^k_{j-m}(a^n_m,\;u_m^n) \right).\\
\end{split}
\]
Let us swap $i$ and $m$ in the first sum and $j$ and $m$ in the second to find
\[
\begin{split}
&\sum_{i,j\in \Z^d} |u_i^{n+1}-u_j^{n+1}|\,K^h_{i-j}\leq \sum_{i,j}\,K^h_{i-j}\, |u_i^n-u_j^n|+{\cal D}_h\\
&\quad+\frac{\delta t}{\delta x}\sum_{i,j,m }\sum_{k=1}^d s_{i,j}^{n}\,(K^h_{m-[1]_k-j}-K^h_{m-j})\,\left(F^k_{m-i}(a^n_i,\;u_i^n) -F^k_{m-i}(a^n_i,\;u_j^n) \right)\\
&\quad+\frac{\delta t}{\delta x}\sum_{i,j,m} \sum_{k=1}^d s_{i,j}^{n}\,(K^h_{i+[1]_k-m}-K^h_{i-m})\,\left(F^k_{m-j}(a^n_j,\;u_i^n) -F^k_{m-j}(a^n_j,\;u_j^n) \right).\\
\end{split}
\]
The moments condition \eqref{momentsflux} on the flux ensures that $F^k_{m-i}$ is small unless $m$ is close to $i$. This will allow us to replace $K^h_{m-[1]_k-j}-K^h_{m-j}$ by $K^h_{i-[1]_k-j}-K^h_{i-j}$ (and similarly for the second sum). More precisely, we write
\[
\begin{split}
& \sum_{i,j,m }\sum_{k=1}^d s_{i,j}^{n}\,(K^h_{m-[1]_k-j}-K^h_{m-j})\,\left(F^k_{m-i}(a^n_i,\;u_i^n) -F^k_{m-i}(a^n_i,\;u_j^n) \right)\\
&\quad=\sum_{i,j,m }\sum_{k=1}^d s_{i,j}^{n}\,(K^h_{i-[1]_k-j}-K^h_{i-j})\,\left(F^k_{m-i}(a^n_i,\;u_i^n) -F^k_{m-i}(a^n_i,\;u_j^n) \right)\\
&\qquad+\sum_{i,j,m }\sum_{k=1}^d s_{i,j}^{n}\,(K^h_{m-[1]_k-j}-K^h_{i-[1]_k-j}-K^h_{m-j}+K^h_{i-j})\\
&\hspace{120pt}\left(F^k_{m-i}(a^n_i,\;u_i^n) -F^k_{m-i}(a^n_i,\;u_j^n) \right).
\end{split}
\]
We recall that
\[
K^h_{i-j}=\frac{\phi(\delta x\,(i-j))}{(\delta x\,|i-j|+h)^d},
\]
so that for a given $0<\gamma\leq 1$
\[\begin{split}
&|K^h_{i-j}-K^h_{m-j}-K^h_{i-[1]_k-j}+K^h_{m-[1]_k-j}|\\
&\qquad\leq \frac{C\,\delta x^{1+\gamma}\,|i-m|^\gamma}{(\delta x\,|i-j|+h)^{d+1+\gamma}}+\frac{C\,\delta x^{1+\gamma}\,|i-m|^\gamma}{(\delta x\,|m-j|+h)^{d+1+\gamma}}.
\end{split}\]
We recall that, because of compact support, we only have to calculate the expression for $|j-i|+|j-m|\leq C\,\delta x^{-1}$. Therefore
\[\begin{split}
&\sum_{m,i\ s.t.|j-i|+|j-m|\leq C\,\delta x^{-1}} |K^h_{i-j}-K^h_{m-j}-K^h_{i-[1]_k-j}+K^h_{m-[1]_k-j}|\,|F^k_{m-i}(a^n_i,\;u_i^n)|\\
&\quad\leq \sum_{i,\ |j-i|\leq C\,\delta x^{-1}} \frac{C\,\delta x^{1+\gamma}}{(\delta x\,|i-j|+h)^{d+1+\gamma}}\, \sum_{m} |i-m|^\gamma\,|F^k_{m-i}(a^n_i,\;u_i^n)|\\
&\qquad+ \sum_{m,\ |j-m|\leq C\,\delta x^{-1}} \frac{C\,\delta x^{1+\gamma}}{(\delta x\,|m-j|+h)^{d+1+\gamma}}\, \sum_{i} |i-m|^\gamma\,|F^k_{m-i}(a^n_i,\;u_i^n)|,\\
\end{split}
\]
giving by \eqref{momentsflux}
\[\begin{split}
&\sum_{m,i\ s.t.|j-i|+|j-m|\leq C\,\delta x^{-1}} |K^h_{i-j}-K^h_{m-j}-K^h_{i-[1]_k-j}+K^h_{m-[1]_k-j}|\,|F^k_{m-i}(a^n_i,\;u_i^n)|\\
&\quad\leq \sum_{i,\ |i-j|\leq C\,\delta x^{-1}} \frac{C\,\delta x^{1+\gamma}}{(\delta x\,|i-j|+h)^{d+1+\gamma}}\,\|f'\|_{L^\infty}\, |a_i^n|\,|u_i^n|\\
&\qquad+\sum_{m,\ |j-m|\leq C\,\delta x^{-1}} \frac{C\,\delta x^{1+\gamma}}{(\delta x\,|m-j|+h)^{d+1+\gamma}}\,\|f'\|_{L^\infty}\,\|a^n\|_{l^p}\,\|u^n\|_{l^{p^*}}.\\
\end{split}
\]
Remark that
\[
\sum_{i\neq 0} \frac{1}{(\delta x\,|i|+h)^{d+1+\gamma}}\leq \frac{C}{h^{1+\gamma}\,\delta x^d}. 
\]
This allows to conclude that 
\[\begin{split}
&\sum_{j,m,i,\ |i-j|+|m-j|\leq C\,\delta x^{-1}} |K^h_{i-j}-K^h_{m-j}-K^h_{i-[1]_k-j}+K^h_{m-[1]_k-j}|\,|F^k_{m-i}(a^n_i,\;u_i^n)|\\
&\qquad\leq C\, \delta x^{-2d}\,\frac{\delta x^{1+\gamma}}{h^{1+\gamma}}\,\|f'\|_{L^\infty}\,\|a^n\|_{l^p}\,\|u^n\|_{l^{p^*}}.
\end{split}
\]
The other terms are handled in the same manner so that
\[
\begin{split}
&\sum_{i,j} |u_i^{n+1}-u_j^{n+1}|\,K^h_{i-j}\\
&\quad\leq \sum_{i,j}\,K^h_{i-j}\, |u_i^n-u_j^n|+{\cal D}_h+C\,\delta t \frac{\delta x^{\gamma-2d}}{h^{1+\gamma}}\,\|f'\|_{L^\infty}\, \|a^n\|_{l^p}\,\|u^n\|_{l^{p^*}}\\
&\quad+\frac{\delta t}{\delta x}\sum_{i,j,m }\sum_{k=1}^d s_{i,j}^{n}\,(K^h_{i-[1]_k-j}-K^h_{i-j})\,\left(F^k_{m-i}(a^n_i,\;u_i^n) -F^k_{m-i}(a^n_i,\;u_j^n) \right)\\
&\quad+\frac{\delta t}{\delta x}\sum_{i,j,m} \sum_{k=1}^d s_{i,j}^{n}\,(K^h_{i+[1]_k-j}-K^h_{i-j})\,\left(F^k_{m-j}(a^n_j,\;u_i^n) -F^k_{m-j}(a^n_j,\;u_j^n) \right).\\
\end{split}
\]
The normalization of the flux, Eq. \eqref{normflux} for instance implies that 
\[
\sum_m F^k_{m-i}(a^n_i,\;u_i^n)=a^{n}_{i,k}\,f(u_i^n),
\] 
where we recall that $a^{n}_{i,k}$ is the $k$ coordinate of the vector $a^n_i$. 
This leads to
\[
\begin{split}
&\sum_{i,j} |u_i^{n+1}-u_j^{n+1}|\,K^h_{i-j}\\
&\quad\leq \sum_{i,j}\,K^h_{i-j}\, |u_i^n-u_j^n|+{\cal D}_h+C\,\delta t \frac{\delta x^{\gamma-2d}}{h^{1+\gamma}}\,\|f'\|_{L^\infty}\, \|a^n\|_{l^p}\,\|u^n\|_{l^{p^*}}\\
&\quad+\frac{\delta t}{\delta x}\sum_{i,j,m }\sum_{k=1}^d s_{i,j}^{n}\,(K^h_{i-[1]_k-j}-K^h_{i-j})\,(a^n_{i,k}-a^n_{j,k})\,(f(u_i^n)- f(u_j^n)).\\
&\quad+\frac{\delta t}{\delta x}\sum_{i,j,m }\sum_{k=1}^d(K^h_{i+[1]_k-j}-2\,K^h_{i-j}+K^h_{i-[1]_k-j})\,s_{i,j}^n\,a^n_{j,k}\,(f(u_i^n)- f(u_j^n)).\end{split}
\]
Observe that provided $h\geq \delta x$
\[
\left|K^h_{i+[1]_k-j}-2\,K^h_{i-j}+K^h_{i-[1]_k-j}\right|\leq \frac{C\,\delta x^2}{(\delta x\,|i-j|+h)^{d+2}}.
\]
Therefore with calculations similar to the previous one, it is possible to show that
\[\begin{split}
&\sum_{i,j,m }\sum_{k=1}^d(K^h_{i+[1]_k-j}-2\,K^h_{i-j}+K^h_{i-[1]_k-j})\,s_{i,j}^n\,a^n_{j,k}\,(f(u_i^n)- f(u_j^n))\\
&\qquad\leq C\,\|f'\|_{L^\infty}\,\|a^n\|_{l^p}\,\|u^n\|_{l^{p^*}}\,\frac{\delta x^2}{h^2}.
\end{split}\]
Therefore since $\gamma\leq 1$, we finally obtain that
\begin{equation}
\begin{split}
&\sum_{i,j} |u_i^{n+1}-u_j^{n+1}|\,K^h_{i-j}\\
&\quad\leq \sum_{i,j}\,K^h_{i-j}\, |u_i^n-u_j^n|+{\cal D}_h+C\,\delta t \frac{\delta x^{\gamma-2d}}{h^{1+\gamma}}\,\|f'\|_{L^\infty}\,\|a^n\|_{l^p}\,\|u^n\|_{l^{p^*}}\\
&\quad+\frac{\delta t}{\delta x}\sum_{i,j,m }\sum_{k=1}^d s_{i,j}^{n}\,(K^h_{i-[1]_k-j}-K^h_{i-j})\,(a^n_{i,k}-a^n_{j,k})\,(f(u_i^n)- f(u_j^n)).\\
\end{split}\label{equivcontinuous}
\end{equation}
We may now bound the last term by using the continuous result of Lemma \ref{lemcommutator}. For this construct continuous fields from the discrete ones.

Consider a set of cubes $C_i$ of size $\delta x$ s.t. $C_i$ is centered at point $x_i$. 

We first define the field $\tilde u^n(x)$ which is piecewise constant with $\tilde u^n(x)=u^n_i$ within $C_i$. We then construct a velocity field $\tilde a^n(x)$ piecewise linear in each cube $C_i$ and such that for any $i,\;j$
\[
\frac{1}{\delta x^{2d}}\,\int_{C_i\times C_j} \nabla K_h(x-y)\cdot (\tilde a^n(x)-\tilde a^n(y))= \sum_{k=1}^d \frac{K^h_{i-[1]_k-j}-K^h_{i-j}}{\delta x}\,(a^n_{i,k}-a^n_{j,k}).
\]
As a consequence, the corresponding norms of $a^n$ and $u^n$ are dominated by the discrete norms
\begin{equation}
\|\tilde u^n\|_{L^q(\R^d)}\leq \|u^n\|_{l^q},\quad \|\tilde a^n\|_{W^{1,p}(\R^d)}\leq \|a^n\|_{d,W^{1,p}},\qquad \|\div \tilde a^n\|_{L^\infty}\leq \sup_i |D^n_i|.\label{continuousdiscretenorms} 
\end{equation}
Furthermore we recall that by Eq. \eqref{defFG}
\[
F(\xi,\zeta)=(f(\xi)-f(\zeta))\,\mbox{sign}(\xi-\zeta),
\]
so that one has the identity
\[\begin{split}
&\delta x^{-1}\,\sum_{i,j,m }\sum_{k=1}^d s_{i,j}^{n}\,(K^h_{i-[1]_k-j}-K^h_{i-j})\,(a^n_{i,k}-a^n_{j,k})\,(f(u_i^n)- f(u_j^n))\\
&\qquad=\delta x^{-2d} \int_{\R^{2d}} \nabla K_h(x-y)\cdot (\tilde a^n(x)-\tilde a^n(y))\,F(u^n(x),u^n(y)).
\end{split}
\]
We may apply Lemma \ref{lemcommutator} to find that provided $\theta\geq 1-1/p$
\[\begin{split}
&\delta x^{-1}\,\sum_{i,j,m }\sum_{k=1}^d s_{i,j}^{n}\,(K^h_{i-[1]_k-j}-K^h_{i-j})\,(a^n_{i,k}-a^n_{j,k})\,(f(u_i^n)- f(u_j^n))\\
&\qquad\leq C\,\delta x^{-2d}\,|\log h|^{\theta}\,\|f'\|_{L^\infty}\,\left(\|\mbox{div}\,\tilde a\|_{L^\infty}\,\|u\|_{1,\theta}+\|\nabla \tilde a^n\|_{B^0_{p,p}}\,\|\tilde u^n\|_{L^{p^*,1}}\right).
\end{split}
\]
Using classical embeddings and the bounds from Eq. \eqref{continuousdiscretenorms}, we estimate that for any $\theta\geq 1-1/p$ and any $q>p^*$
\[\begin{split}
&\delta x^{-1}\,\sum_{i,j,m }\sum_{k=1}^d s_{i,j}^{n}\,(K^h_{i-[1]_k-j}-K^h_{i-j})\,(a^n_{i,k}-a^n_{j,k})\,(f(u_i^n)- f(u_j^n))\\
&\qquad\leq C\,\delta x^{-2d}\,|\log h|^{\theta}\,\|f'\|_{L^\infty}\,\left(\sup_i |D_i^n|\,\|u^n\|_{d,1,\theta}+\|a^n\|_{d,W^{1,p}}\,\|u^n\|_{l^{q}}\right).
\end{split}
\]
If we combine this bound together with the bound \eqref{boundDterm} on ${\cal D}_h$ into \eqref{equivcontinuous}
\begin{equation}
\begin{split}
& \|u^{n+1}\|_{\alpha,1,\theta}=\sup_{h\geq \delta x^\alpha} |\log h|^{-\theta}\,\delta x^{2d}\,\sum_{i,j} K^h_{i-j}\,|u_i^{n+1}-u_j^{n+1}|\\
&\quad\leq \|u^{n}\|_{\alpha,1,\theta}+C\,\delta t\,\frac{\delta x^{\gamma}}{h^{1+\gamma}}\,\|f'\|_{L^\infty}\,\|a^n\|_{l^p}\,\|u^n\|_{l^{p^*}}\\
&\qquad+\delta t\,\|f'\|_{L^{\infty}}\,\left(\|u^n\|_{l^{p^*}}\,\|D^n\|_{\alpha,p,p(\theta-1/p^*)} +\|D^n\|_{l^\infty}\,\|u^n\|_{d,1,\theta}\right)\\
&\qquad+\delta t\,\|f'\|_{L^\infty}\,\left( \|D^n\|_{l^\infty}\,\|u^n\|_{\alpha,1,\theta}+\|a^n\|_{d,W^{1,p}}\,\|u^n\|_{l^{q}}\right).
\end{split}\label{finalgronwall}
\end{equation}
A discrete Gronwall estimate allows to conclude the proof.
\section*{Appendix A: Proof of Lemma \ref{lem:convolution}.\label{appendix}}
The estimates presented here are classical and we refer for instance to \cite{Ab}, \cite{BaChDa} or \cite{Stein2}.

   Choose any family $\Psi_k\in {\cal S}(\R^d)$ s.t.
\begin{itemize}
\item For $k\geq 1$, its Fourier transform $\hat \Psi_k$ is positive and compactly supported in the annulus $\{2^{k-1}\leq |\xi|\leq 2^{k+1}\}$.
\item It leads to a decomposition of the identity in the sense that there exists $\Psi_0$ with $\hat \Psi_0$ compactly supported in $\{|\xi|\leq 2\}$ s.t. for any $\xi$
\[
1=\sum_{k\geq 0} \hat \Psi_k(\xi).
\]
\item The family is localized in $\R^d$ in the sense that for all $s>0$
\[
\sup_k \|\Psi_k\|_{L^1}<\infty,\quad \sup_k 2^{ks}\,\int_{\R^d} |z|^s\,|\Psi_k(z)|\,dz<\infty.
\] 
\end{itemize}
Such a family can be used to define the usual Besov norms with
\begin{equation}
\|u\|_{B^{s}_{p,q}}=\left\| 2^{s\,k}\,\|\Psi_k\star u\|_{L^p_x}\right\|_{l^q_k}=\left(\sum_{k=0}^\infty 2^{s\,k\,q}\,\|\Psi_k\star u\|_{L^p_x}^q\right)^{1/q}<\infty.\label{defbesov}
\end{equation}
For this reason it is useful to denote
\[
U_{k}=\Psi_{k}*u.
\]
Since $U_k$ is localized in frequency, one may easily relate all its Sobolev norms: For any $1<p<\infty$, any $k\geq 1$ and any $\alpha$
\begin{equation}
\|U_k\|_{W^{\alpha,p}}\leq C_p\,2^{k\alpha}\,\|U_k\|_{L^p}.\label{soblp}
\end{equation}

\bigskip

We first give the bound that we use for $k\leq |\log_2 r|$. Since the kernel $L$ has 0 average 
\[
L_{r}\star U_{k}=\int_{\mathbb{R}^{d}}L_{r}\left(x-y\right)\left(U_{k}\left(y\right)-U_{k}\left(x\right)\right)dy.
\]
Therefore 
\begin{eqnarray*}
\left\Vert L_{r}\star U_k\right\Vert _{L^{p}}\leq\int_{\mathbb{R}^{d}}L_{r}\left(z\right)\left\Vert U_{k}\left(.\right)-U_{k}\left(.+z\right)\right\Vert _{L^{p}}dz\\
\leq\int_{\mathbb{R}^{d}}L_{r}\left(z\right)\left|z\right|^{s}\left\Vert U_{k}\right\Vert _{W^{s,p}}dz.
\end{eqnarray*}
Since $L$ has bounded moments then $\int |z|^s L_r(z)\,dz=r^s\,\int |z|^s\,L(z)\,dz$,
yielding 
\begin{equation}
\left\Vert L_{r}\star U_k\right\Vert _{L^{p}}\leq Cr^{s}2^{ks}\left\Vert U_{k}\right\Vert _{L^{p}}\label{eq:conv2-1}
\end{equation}
by inequality \eqref{soblp} for $\alpha=s$, and for a fixed constant $C$ depending only on $\int\left|z\right|^{s}L\left(z\right)dz.$

\smallskip

For the case $k\geq |\log_2 r|$, we use that $L\in W^{s,1}$ and deduce that
\begin{equation}
\left\Vert L_{r}\star U_k\right\Vert _{L^{p}}\leq\left\Vert L_{r}\right\Vert _{W^{s,1}}\left\Vert U_{k}\right\Vert _{W^{-s,p}}\leq Cr^{-s}2^{-ks}\left\Vert U_{k}\right\Vert _{L^{p}},\label{eq:conv3}
\end{equation}
by using again \eqref{soblp} but for $\alpha=-s$, where $C$ only depends on the $W^{s,1}$ norm of $L$. 

Using now this decomposition and the two bounds, \eqref{eq:conv2-1}-\eqref{eq:conv3}
\[
\begin{split}
\int_{h_{0}}^{1}\left\Vert L_{r}\star u\right\Vert _{L^{p}}\frac{dr}{r}&=\sum_{k=0}^{\infty}\int_{h_{0}}^{1}\left\Vert L_{r}\star U_{k}\right\Vert _{L^{p}}\frac{dr}{r}\\
\leq &C\,\sum_{k=0}^{\infty}\left\Vert U_{k}\right\Vert _{L^{p}}\left(\int_{h_{0}}^{2^{-k}}r^{s}2^{ks}\frac{dr}{r}
+\int_{\max\left(h_{0},2^{-k}\right)}r^{-s}2^{-ks}\frac{dr}{r}\right).
\end{split}
\]
This implies
that 
\begin{eqnarray}
\int_{h_{0}}^{1}\left\Vert L_{r}*u\right\Vert _{L^{p}}\frac{dr}{r}\leq C\sum_{k\leq\left|\log_{2}h_{0}\right|}^{\infty}\left\Vert U_{k}\right\Vert _{L^{p}}\label{eq:conv4}
+C\sum_{k>\left|\log_{2}h_{0}\right|}^{\infty}\frac{2^{-ks}}{h_{0}^{s}}\left\Vert U_{k}\right\Vert _{L^{p}}.\nonumber 
\end{eqnarray}
Now simply bound
\[
\begin{split}
\sum_{k\leq\left|\log_{2}h_{0}\right|}^{\infty}\left\Vert U_{k}\right\Vert _{L^{p}}+\sum_{k>\left|\log_{2}h_{0}\right|}^{\infty}\frac{2^{-ks}}{h_{0}^s}\left\Vert U_{k}\right\Vert _{L^{p}}&\leq C\sum_{0}^{\infty}\left\Vert U_{k}\right\Vert _{L^{p}}\\
&=\,C\left\Vert u\right\Vert _{B_{p,1}^{0},}
\end{split}
\]
which gives (\ref{eq:conv1}) in the case $q=1$.

Next remark that
\begin{eqnarray*}
\sum_{k>\left|\log_{2}h_{0}\right|}^{\infty}\frac{2^{-ks}}{h_{0}^s}\left\Vert U_{k}\right\Vert _{L^{p}}\leq\sup\left\Vert U_{k}\right\Vert _{L^{p}}\sum_{k>\left|\log_{2}h_{0}\right|}^{\infty}\frac{2^{-k\,s}}{h_{0}^s}\leq\sup\left\Vert U_{k}\right\Vert _{L^{p}}\\
\leq C\left\Vert u\right\Vert _{B_{p,\infty}^{0}}.
\end{eqnarray*}
On the other hand,
\[
\sum_{k\leq\left|\log_{2}h_{0}\right|}^{\infty}\left\Vert U_{k}\right\Vert _{L^{p}}\leq \left|\log_{2}h_{0}\right|^{1-1/q}\,\left(\sum_k \|U_k\|_{L^p}^q\right)^{1/q}\leq \left|\log_{2}h_{0}\right|^{1-1/q}\,\|u\|_{B^0_{p,q}},
\]
implying \eqref{eq:conv1} for general $q$.

We now recall the well-known embedding of $L^p$ into $B^0_{p,2}$ when $p\leq 2$, giving
\[
\sum_{k\leq\left|\log_{2}h_{0}\right|}^{\infty}\left\Vert U_{k}\right\Vert _{L^{p}}\leq C\sqrt{\left|\log_{2}h_{0}\right|}\left\Vert u\right\Vert _{L^{p}}.
\]
Therefore (\ref{eq:conv4}) yields 
\[
\int_{h_{0}}^{1}\left\Vert L_{r}*u\right\Vert _{L^{p}}\frac{dr}{r}\leq C\sqrt{\left|\log_{2}h_{0}\right|}\left\Vert u\right\Vert _{L^{p}}+C\left\Vert u\right\Vert _{B_{p,\infty}^{0}},
\]
which proves (\ref{eq:conv2}).

\end{document}